\documentclass[12pt]{article}      
\usepackage{graphicx}

\usepackage{graphicx}
\usepackage{url}
\usepackage{color}
\definecolor{maroon}{rgb}{.69,.188,.376}
\definecolor{darkgreen}{rgb}{0,.5,0}
\definecolor{darkblue}{rgb}{0,0,.5}
\definecolor{magenta}{rgb}{1,0,1}

\usepackage[mathscr]{euscript}		

\usepackage{dsfont}

\usepackage{psfrag}			
\usepackage[colorlinks=true]{hyperref}
\hypersetup{pdftex, colorlinks=true, linkcolor=maroon, citecolor=maroon,
  filecolor=blue,urlcolor=blue}
\usepackage{amsmath}
\usepackage{amsthm}

\usepackage{amssymb}

\usepackage{caption} 

\usepackage{anysize}

\usepackage{enumerate} 
\usepackage{enumitem}

\usepackage[top=.95in, bottom = 1 in, left=1in, right = .6in]{geometry}

\newtheorem{theorem}{Theorem}[section]
\newtheorem{lemma}[theorem]{Lemma}
\newtheorem{corollary}[theorem]{Corollary}
\newtheorem{proposition}[theorem]{Proposition}
\newtheorem{remark}[]{Remark}

\newtheorem{definition}[]{Definition}

\numberwithin{equation}{section}

\definecolor{Red}{rgb}{1,0,0}
\definecolor{Blue}{rgb}{0,0,1}
\definecolor{Olive}{rgb}{0.41,0.55,0.13}
\definecolor{Yarok}{rgb}{0,0.5,0}
\definecolor{Green}{rgb}{0,1,0}
\definecolor{MGreen}{rgb}{0,0.8,0}
\definecolor{DGreen}{rgb}{0,0.55,0}
\definecolor{Yellow}{rgb}{1,1,0}
\definecolor{Cyan}{rgb}{0,1,1}
\definecolor{Magenta}{rgb}{1,0,1}
\definecolor{Orange}{rgb}{1,.5,0}
\definecolor{Violet}{rgb}{.5,0,.5}
\definecolor{Purple}{rgb}{.75,0,.25}
\definecolor{Brown}{rgb}{.75,.5,.25}
\definecolor{Grey}{rgb}{.7,.7,.7}
\definecolor{Black}{rgb}{0,0,0}

\newcommand{\ignore}[1]{{}}

\date{\today}

\begin{document}

\title{Complex Solutions to Bessel SDEs and SLEs}
\author{Atul Shekhar%
  \thanks{University Lyon 1.
    Email: \url{atulshekhar83@gmail.com}}
\and Vlad Margarint
\thanks{NYU Shanghai.
    Email: \url{margarint@nyu.edu}} 
 }


\maketitle
\begin{abstract}
We consider a variant of Bessel SDE by allowing the solution to be complex valued. Such SDEs appear naturally while studying the trace of Schramm-Loewner-Evolutions (SLE). 
We establish the existence and uniqueness of the strong solution to such SDEs when the dimension is negative. We also consider the stochastic flow associated to such SDEs and prove that it is almost surely continuous. Our proofs are based on an improvement of the derivative estimate of Rohde-Schramm \cite{RS05}. We finally show the connection between such stochastic flows and SLE$_{\kappa}$ for $\kappa <4$.
\end{abstract}

\section{Introduction and results.} \label{intro}

In this article we study a complex variant of Bessel stochastic differential equation (SDE). Such complex Bessel processes appear naturally in the study of Schramm-Loewner-Evolutions SLE$_{\kappa}$, $\kappa \in (0,4)$, see Corollary \ref{embedding} below.

\subsection{Real Bessel Processes}
Let us first recall some basic facts on classical real valued Bessel processes. There are various ways to define it and we will follow the approach of \cite[Chapter-11]{RY}. Also see \cite{lawler-notes} for a different approach based on Girsanov transformation. \\
The content of this subsection is very well known and readers familiar with Bessel processes can skip to next section. However, we believe that recalling the following basic facts helps the presentation of our paper and clarify key points in our discussion. \\

Let $(\Omega, \mathcal{F}, \mathbb{P})$ be a complete probability space and $B$ be a one dimensional Brownian motion defined on $\Omega$ starting from $B_0= 0$ with its natural filtration $\{\mathcal{F}_t\}_{t\geq 0}$. For $\delta\geq 0$ and $x\geq 0$, a $\delta$-dimensional Bessel process started at $x$ is defined as the real valued solution to SDE
\begin{equation}\label{real-bessel}
dX_t = dB_t + \frac{\delta -1}{2}\frac{1}{X_t}dt, \hspace{2mm}X_0 = x.
\end{equation}
For $x\neq 0$, equation \eqref{real-bessel} admits a unique strong solution $X_t$ for $t < T^x$, where $T^x$ is the first hitting time of level 0 starting from $x$ (actually $X_t-B_t$ satisfies an ordinary differential equation (ODE) and we will sometimes slightly incorrectly say that $X_t$ satisfies an ODE when $x\neq 0$). It is easy to see that $X_t \to 0$ as $t\uparrow T^x$ whenever $T^x < \infty$. For  $x=0$, equation \eqref{real-bessel} is a singular equation. If $\delta >1$, one can make sense of \eqref{real-bessel} by imposing an additional condition that $1/|X_t|$ is a Lebesgue integrable function so that the right hand side of \eqref{real-bessel} is well defined. It was proven in \cite{Mckean-Jr} that there exists a unique non-negative solution $X$ to \eqref{real-bessel} which is defined as the $\delta$-dimensional Bessel process started at $0$. Note that since $-B$ is also a Brownian motion, there is also a non-positive solution to \eqref{real-bessel} which can be obtained by reflecting the non-negative solution. We will interpret this selection of non-negative solution as choosing a continuous branch out of many solutions. The case of $\delta \in [0,1]$ requires a different definition because \eqref{real-bessel} doesn't admit any solution such that $1/|X_t|$ is Lebesgue integrable. In fact, $X_t$ is not a semimartingale for $\delta \in [0,1)$. The following alternative definition works well for all $\delta \geq 0$ and coincides with the previous definition for $\delta >1$. Consider squared Bessel processes defined by SDE 
\begin{equation}\label{sq-eqn}
dZ_t = 2\sqrt{|Z_t|}dB_t + \delta dt, \hspace{2mm} Z_0 = x^2.
\end{equation}

Since square root function on $[0,\infty)$ is a $1/2$-H\"older function, Yamada-Watanabe Theorem implies that \eqref{sq-eqn} admits a unique strong solution. The $\delta$-dimensional Bessel process is then defined by $X := \sqrt{|Z|}$. It follows by stochastic comparison principles that for $\delta\geq 0$, $Z_t \geq 0$. Thus $X= \sqrt{|Z|}= \sqrt{Z}$ and \eqref{sq-eqn} is equivalent to 
\begin{equation}\label{sq-eqn-2}
dZ_t = 2\sqrt{Z_t}dB_t + \delta dt, \hspace{2mm} Z_0 = x^2.
\end{equation}  
The choice of non-negative solution above can also be intuitively viewed as the reflected solution, i.e. the solution which reflects back towards positive axis whenever it hits zero.

The above definition of $\delta$-dimensional Bessel processes using \eqref{sq-eqn} is also valid when $\delta< 0$. In this case the solution starting at zero will be non-positive and the modulus inside square root function is required to make sense of \eqref{sq-eqn} as a real equation. In some sense, we are forcing the solution $Z$ to be real valued by putting a modulus inside the square root function. We will show in this article that there are other interesting ways to continue the solution after it has hit zero. \\

\subsection{Complex Variants of Bessel Processes}

Our main idea in this article is to allow solutions to Bessel SDEs \eqref{real-bessel} to be complex valued and consider a variant of \eqref{sq-eqn} using complex square root. Consider the singular equation 
\begin{equation}\label{complex-bessel}
dH_t =  dB_t +\frac{\delta-1}{2}\frac{1}{H_t}dt, \hspace{2mm}H_0 = 0,
\end{equation}
where we allow the solution $H$ to be complex valued. Similarly as above, we will use the idea of considering the squared equation  to make sense of \eqref{complex-bessel}. To this end, consider the square map $z \mapsto z^2$ defined on $\mathbb{C} \to \mathbb{C}$, where $\mathbb{C}$ is the complex plane. Let $\mathbb{H} := \{z \in \mathbb{C}| Im(z) >0\}$ denote the upper half plane. Since the square root map on complex plane is multivalued, we will work with the branch $z\mapsto \sqrt{z}$ defined on $\mathbb{C}\setminus (0,\infty) \to \mathbb{H}\cup\{0\}$ defined by 
\[ \sqrt{z} = sgn(Im(z))\sqrt{\frac{|z|+ Re(z)}{2}} + i\sqrt{\frac{|z|- Re(z)}{2}},\]
where $sgn(y)= 1,-1,0$ for $y>0, y<0, y=0$ respectively. With an abuse of notation, we also use $\sqrt{x}$ to mean the usual real square root of $x$ for $x\geq 0$. We will need the following definition of square roots of continuous $\mathbb{C}$-valued curves. The following definition also appeared in \cite{shekhar-tran-wang}. Results of \cite{shekhar-tran-wang} can be considered as deterministic analogs of results in the present article.


\begin{definition}\label{branch-def}
For a continuous curve (resp. continuous adapted process) $Y:[0,\infty)\to \mathbb{C}$, a branch square root of $Y$ is a measurable (resp. adapted) curve $A:[0,\infty) \to \overline{\mathbb{H}}$ such that $A_t^2 = X_t $ for all $t\in[0,\infty)$. We then write $A = \sqrt{Y}^b$.
\end{definition}
Note that there are two possible extensions of $\sqrt{z}$ function as $z\to x \in (0,\infty)$, $+\sqrt{x}$ or $-\sqrt{x}$ depending on whether $x$ is approached from upper half plane or lower half plane respectively. Thus, there could be more than one branch square roots for a given continuous process $Y$. Choosing a branch square root is equivalent to making a choice from these two possible extensions in a measurable/adapted way whenever $Y$ hits $(0,\infty)$. Also, there is no loss of generality in working with the upper half plane in the above definition instead of lower half plane. This will be akin to the conventional choice of non-negative solutions to real Bessel SDEs \eqref{real-bessel} as mentioned above. \\

We now consider the complex squared Bessel SDE given by  
\begin{equation}\label{squared-bessel}
dY_t = 2\sqrt{Y_t}^bdB_t + \delta dt, \hspace{2mm}Y_0=0,
\end{equation}
where $\sqrt{Y}^b$ is some branch square root of $Y$. Note that since branch squared roots are assumed to be adapted to filtration $\{\mathcal{F}_t\}_{t\geq 0}$, equation \eqref{squared-bessel} is well defined as an It\^o SDE. The choice of branch square root $\sqrt{Y}^b$ may a priori depend on the solution $Y$ itself. This is the principal difference of equation \eqref{squared-bessel} as compared to \eqref{sq-eqn} where the choice of branch square root was fixed beforehand. Equation \eqref{squared-bessel} is more natural to consider if we want solution $Y$ to depend ``holomorphically" on the initial datum because unlike $\sqrt{|z|}$ function, $\sqrt{z}$ is a holomorphic function. A natural question is whether the equation \eqref{squared-bessel} admits a unique solution? Our first main result is the following theorem establishing the existence and uniqueness of strong solution to \eqref{squared-bessel}.

\begin{theorem}\label{main-thm-1}
For $\delta <0$,
\begin{enumerate}
\item If  \hspace{1mm}$Y$ is a solution to \eqref{squared-bessel}, then almost surely for all $t >0$, $Y_t \in \mathbb{C}\setminus [0,\infty)$. In particular, $\sqrt{Y}^b_t = \sqrt{Y_t}$ and \eqref{squared-bessel} is equivalent to 
\begin{equation}\label{equivalent-eqn}
dY_t = 2\sqrt{Y_t}dB_t + \delta dt, \hspace{2mm} Y_0 = 0.
\end{equation}
\item There exists a continuous adapted process $Y$ satisfying \eqref{equivalent-eqn}. Also, if $Y$ and $\tilde{Y}$ are two continuous adapted solutions to \eqref{equivalent-eqn}, then 
\begin{equation}\label{equality}
\mathbb{P}[Y_t = \tilde{Y}_t \hspace{2mm} \mbox{for all $t\geq 0$}] =1.
\end{equation}
\end{enumerate}
\end{theorem}

The unique strong solutions $Y$ obtained in Theorem \ref{main-thm-1} will be called complex squared Bessel process of dimension $\delta <0$ started at $0$. Following [\cite{RY}-Chapter$11$], we will abbreviate it by $CBESQ^{\delta}(0)$. We define the solution $H$ to equation \eqref{complex-bessel} by $H_t := \sqrt{Y_t}$. We call $H$ the complex Bessel process of dimension $\delta < 0$ started at $0$ abbreviated by $CBES^{\delta}(0)$. A similar half plane valued solutions to Bessel SDEs has also been considered in \cite[Proposition 3.8]{D-M-S}, but they have only proved the existence and uniqueness of weak solutions. \\

\begin{remark}
We do not yet know for sure whether the solution $H$ is a semimartingale and whether it satisfies 
\[\int_0^t \frac{1}{|H_r|}dr < \infty \hspace{2mm}?\]

Comparing with real Bessel processes with dimension $\delta>0$ suggests that the above integral is finite at least for large $|\delta|$.
\end{remark}

The proof of Theorem \ref{main-thm-1} will be based on derivative estimates of Rohde-Schramm obtained in \cite{RS05}. Since $\sqrt{z}$ function on $\mathbb{C}\setminus (0,\infty)$ is not a $1/2$-H\"older function, Yamada-Watanabe Theorem does not apply. The basic idea behind proof of Theorem \ref{main-thm-1} is that the negative drift present in the equation \eqref{squared-bessel} will push the solution $Y$ away from the non-negative real axis and $\sqrt{Y}$ escapes to upper half plane. We then use the derivative estimates obtained in \cite{RS05} to conclude the proof, see Section \ref{main-proof} for details. When $\delta = 0$, equation \eqref{squared-bessel} doesn't have unique solution. One trivial solution is $Y\equiv 0$. One can also construct non-zero solutions $Y$ to \eqref{squared-bessel} by examining the SLE$_4$. We believe that this is very closely related to the work of Bass-Burdzy-Chen \cite{bass-chen} where they prove the uniqueness of strong solution to certain degenerate real SDEs under the assumption that the solution spends zero time at zero. Theorem \ref{main-thm-1} is also closely related to work of Krylov-R\"ockner \cite{krylov-rockner} which considers multidimensional SDEs with singular drifts. Equation \eqref{complex-bessel} can be viewed as a two dimensional SDE with singular drift.  A distinction between Theorem \ref{main-thm-1} and results in \cite{krylov-rockner} is that the noise term $B$ in \eqref{complex-bessel} is only one dimensional, see also \cite{smoothing-boundary} for a related work. \\

We also consider the stochastic flow associated to \eqref{complex-bessel} on the real line $\mathbb{R}$. More precisely, for each $ (s,t)\in \Delta :=\{(s,t)| 0\leq s \leq t < \infty\}$ and $x \in \mathbb{R}$, define  $H(s,t,x)$ as the solution to equation 
\begin{equation}\label{flow-eqn}
dH(s,t,x) =  dB_t + \frac{\delta -1}{2}\frac{1}{H(s,t,x)}dt, \hspace{2mm} H(s,s,x) = x.
\end{equation} 
When $x \neq 0$, \eqref{flow-eqn} admits a unique strong solution for $t\leq  T^{s,x}$, where $T^{s,x}$ is the first time the solution hits zero. For $t >T^{s,x}$, we use the strong Markov property of Brownian motion and define $H(s,t,x)= H(T^{s,x},t,0)$, where $H(T^{s,x},t,0)$ is taken to be the half plane solution as constructed in Theorem \ref{main-thm-1}. We thus obtain a random field $H= \{H(s,t,x)\}_{(s,t,x)\in \Delta\times \mathbb{R}}$ which is called the stochastic flow associated with equation the \eqref{complex-bessel}. Our second main result is the following theorem on the existence of a continuous modification of the field $H$.

\begin{theorem}\label{cont-flow}
There exists a modification $\tilde{H}$ of the field $H$ which is almost surely jointly continuous in $(s,t,x)\in \Delta \times \mathbb{R}$.
\end{theorem}

The motivation for defining $CBES^{\delta}(0)$ processes comes from SLE$_{\kappa}$ curves $\kappa \in (0,4)$, see \cite{lawler-book} for a detailed introduction to SLEs. It was proven in \cite{RS05} that SLE$_{\kappa}$ curve $\gamma$ exists and $\gamma$ is simple for $\kappa \leq 4$. Let $\gamma$ be SLE$_{\kappa}$ for $\kappa <4$ with the driving Brownian motion $\sqrt{\kappa}W_t$. If $\mathbf{H}_t := \mathbb{H}\setminus \gamma[0,t]$ and $f_t: \mathbb{H}\to \mathbf{H}_t$ be the conformal map such that $f_t(z) = z + O(1)$ as $z\to \infty$,  it is well known (see e.g. Lemma $2.1$ in \cite{shekhar-tran-wang}) that for $z\in \mathbb{H}$, $f_t(z + \sqrt{\kappa}W_t) = \sqrt{\kappa}h_t(z)$, where for $s\in [0,t]$ $h_s(z)$ solves the equation 
\begin{equation*}
dh_s(z) = dW_s^t + \frac{\delta-1}{2}\frac{1}{h_s(z)}ds, \hspace{2mm} h_0(z) = z\in \mathbb{H}
\end{equation*}
with $W_s^t = W_t - W_{t-s}$ and $\delta = 1-\frac{4}{\kappa}$. Let $B_t = W_1 - W_{1-t}$ be the time reversed Brownian motion. Then it follows easily after simple manipulations that $h_{u-1+t}(z)= H(1-t,u,z)$. It was proven in \cite{RS05} that 
\begin{equation}\label{RS-limit}
\gamma_t = \lim_{y\to 0+}f_t(iy + \sqrt{\kappa}W_t).
\end{equation}
From the uniqueness of strong solution to \eqref{squared-bessel}, it follows easily that $H(1-t,u,z) \to H(1-t,u,0)$ as $z\to 0$. In particular, it implies $\gamma_t = \sqrt{\kappa}H(1-t, 1, 0) = \sqrt{\kappa}\tilde{H}(1-t, 1, 0)$. Since $\gamma$ and $\tilde{H}$ are almost surely continuous, we obtain the following corollary. 

\begin{corollary}[SLEs as Stochastic Flows]\label{embedding}
For $\kappa \in (0,4)$ and $\delta = 1-\frac{4}{\kappa} $, the process $\{\sqrt{\kappa}\tilde{H}(1-t,1, 0)\}_{t\in [0,1]}$ has the same law as the chordal SLE$_{\kappa}$ in $\mathbb{H}$ restricted on the unit time interval $[0,1]$. 
\end{corollary}

As we can see in the above argument, the $CBES^{\delta}(0)$ process was constructed precisely to give a characterization of the limit \eqref{RS-limit} by giving a canonical self contained meaning to equation 
\eqref{complex-bessel} started from zero. In the limit \eqref{RS-limit}, the point zero is approached only from vertical direction (non-tangential limit). However, the point zero does not distinguish between different rays in $\mathbb{H}$ approaching to zero at different angles (tangential limit) and the uniqueness of solution to \eqref{complex-bessel} started from zero is slightly stronger than the existence of the limit \eqref{RS-limit}. The uniqueness of solution to \eqref{complex-bessel} started from zero implies that we do have equivalence of non-tangential and tangential limit. This is no longer true in general situations. \\ 

We do not yet have a description of SLE$_{\kappa}$ using a SDE of type \eqref{squared-bessel} for $\kappa \geq 4$ or $\delta\geq 0$ because the uniqueness of solution to \eqref{squared-bessel} fails. The equivalency between \eqref{squared-bessel} and \eqref{equivalent-eqn} is also no longer true for $\delta \geq 0$. For $\kappa >4$, it is known that SLE$_{\kappa}$ is a non-simple curve. Thus, the solution $Y$ to \eqref{squared-bessel} which describes SLE$_{\kappa}$ for $\kappa >4$ will stay on positive real axis with positive probability and it is interesting to ask which branch square root of $Y$ appearing in \eqref{squared-bessel} is suitable for such cases. We plan to investigate further in this direction in our future projects.  \\


\textbf{Acknowledgments:} AS would like to thank Yilin Wang and Christophe Garban for various fruitful discussions. VM acknowledges the support of NYU-ECNU Institute of Mathematical Sciences at NYU Shanghai.

\section{Proof of Theorem \ref{main-thm-1}.}\label{main-proof}

For the proof of the Theorem \ref{main-thm-1}, it will be beneficial to consider equation \eqref{flow-eqn} started from $z\in \mathbb{H}$. Note that when $z\in \mathbb{H}$, $Im(H(s,t,z))$ is strictly increasing in $t$ and it stays positive. Thus, \eqref{flow-eqn} admits a unique ODE solution $H(s,t,z)$ for all time $t\geq s$. The proof of Theorem \ref{main-thm-1} will be based on following proposition. Let $H'(s,t,iy) = \partial_y H(s,t,iy)$. For a continuous martingale $M$, $[M]$ denotes its quadratic variation process.

\begin{proposition}\label{moment-prop}For some constant $\lambda>2$ depending only on $\delta<0$, the following holds:
\begin{enumerate}
\item For each fixed $s\geq 0$ and $y>0$, there exists a continuous martingale $\{M_t\}_{t\geq s}$ adapted to the filtration of $\{B_t-B_s\}_{t\geq s}$ such that $M_s=0$, and  
\begin{equation}\label{exp-martingale}
|H'(s,t,iy)|^{\lambda}\leq \exp\biggl\{M_t - \frac{1}{2}[M]_t\biggr\}.
\end{equation}
In particular, for all $K\geq 1$ and $T>0$,
\begin{equation}\label{weak-tail}
\mathbb{P}\bigl[\sup_{t\in[s,T]} |H'(s,t,iy)| \geq  K\bigr] \leq \frac{1}{K^{\lambda}}.
\end{equation}

\item For all $T> 0$, almost surely there exist constants $C(\omega,T)$ and $\beta \in (0,1)$ depending only on $\delta<0$ such that
\begin{equation}\label{crucial-estimate}
\sup_{0\leq s \leq t\leq T}|H'(s,t,iy)| \leq C(\omega, T)y^{-\beta} \hspace{2mm} \mbox{for all}\hspace{2mm}y \in (0,1].
\end{equation}
\end{enumerate}
\end{proposition}

An estimate similar to \eqref{crucial-estimate} has also been obtained in \cite{RS05}, but the estimate in \cite{RS05} is not uniform in $t$ as compared to \eqref{crucial-estimate}. The proof of Proposition \ref{moment-prop} is deferred until section \ref{moment-section}.

\begin{proof}[Proof of Theorem \ref{main-thm-1}-(a)]
Let 
\[\tau = \inf\{t >0 | Y_t \in \mathbb{C}\setminus [0, \infty)\}.\]
Then $\tau$ and $\tau_n = \tau \wedge \frac{1}{n}$ are stopping times. Also note that since $|\sqrt{Y_t}^b| = \sqrt{|Y_t|}$, using Burkholder-Davis-Gundy and Cauchy-Schwarz inequality it follows that for some constant $C < \infty$
\begin{equation}
\mathbb{E}[\sup_{t\in [0,T]}|Y_t|] \leq C(1 + \sqrt{\mathbb{E}[\sup_{t\in [0,T]}|Y_t|]}).
\end{equation}
It implies $\mathbb{E}[\sup_{t\in [0,T]}|Y_t|] < \infty$ and $Y_t -\delta t$ is a true martingale. Thus by Doob's optional sampling theorem $\mathbb{E}[Y_{\tau_n}] -\delta \mathbb{E}[\tau_n] = 0$. Also, on the event $\{\tau > 0\}$, $Y_t \in [0, \infty)$ for all $t \in [0, \tau]$ and thus $\mathbb{E}[Y_{\tau_n}] \geq 0$. But since $\delta <0$, it implies $\mathbb{E}[\tau_n] = 0$. Consequently $\tau = 0$. Since solution to \eqref{flow-eqn} with starting point $z\in \mathbb{H}$ stays in $\mathbb{H}$, if follows that $Y_t \in \mathbb{C}\setminus [0, \infty)$ for $t >0$.\\

\end{proof}

\begin{proof}[Proof of Theorem \ref{main-thm-1}-(b)] 
If $Y, \tilde{Y}$ satisfy \eqref{equivalent-eqn}, then $Im(Y_t)$ and $Im(\tilde{Y}_t)$ are both martingales starting from $0$ with quadratic variation processes given by
\[ [Im(Y)]_t = 4 \int_0^t Im(\sqrt{Y}_r)^2dr, \hspace{2mm}[Im(\tilde{Y})]_t = 4 \int_0^t Im(\sqrt{\tilde{Y}}_r)^2dr.\]
Note that since both $Im(\sqrt{Y}_t)$ and $Im(\sqrt{\tilde{Y}_t})$ are strictly increasing process, $[Im(Y)]_t$ and $[Im(\tilde{Y})]_t$ are strictly increasing as well. Since martingales are time change of Brownian motion and zero set of Brownian motion has no isolated points, there exist sequence $s_n, \tilde{s}_n \to 0+$ such that $Im(Y_{s_n}) = Im(\tilde{Y}_{\tilde{s}_n}) =0$, i.e. $Y_{s_n}, \tilde{Y}_{\tilde{s}_n} \in (-\infty, 0)$. Using the flow property, for $t>0$ and $n$ large enough, $\sqrt{Y_t} = H(s_n, t, \sqrt{Y_{s_n}})$ and $\sqrt{\tilde{Y}_t} = H(\tilde{s}_n, t, \sqrt{\tilde{Y}_{\tilde{s}_n}})$. We now claim that almost surely 
\begin{equation}\label{crucial-limit}
\lim\limits_{(s,y) \to (0+, 0+)}H(s,t, iy)
\end{equation}
exists. Assuming this and the fact that $\sqrt{Y_{s_n}},\sqrt{\tilde{Y}_{\tilde{s}_n}}$ are purely imaginary and tends to $0$ as $n\to \infty$ easily implies $Y_t = \tilde{Y_t}$ almost surely. Since $Y, \tilde{Y}$ are continuous processes, claim \eqref{equality} follows easily. For proving the existence of limit \eqref{crucial-limit} we use the estimate \eqref{crucial-estimate}. For $y < \tilde{y}$,
\begin{equation}
|H(s,t,iy) - H(s, t, i\tilde{y})| \leq \int_y^{\tilde{y}}|H'(s,t,ir)|dr \leq C (\tilde{y}^{1-\beta} - y^{1-\beta}). 
\end{equation} 
Thus $H(s,t, iy)$ converges uniformly in $s,t$ as $y \to 0+$ to some function continuous in $s,t$. This implies the limit \eqref{crucial-limit} exists which completes the proof.  \\
For the existence of a solution to \eqref{squared-bessel}, note that using the same argument as above, almost surely $H(0,t,0+):= \lim\limits_{y \to 0+}H(0,t,iy)$ exists uniformly in $t$. Let $Y_t := H(0,t,0+)^2$. Note that 
\[H(0,t,iy)^2 = -y^2 + 2\int_0^t H(0,r,iy)dB_r  + \delta t.\]
Letting $y \to 0+$ in above and using the dominated convergence theorem for stochastic integrals implies that $Y$ satisfies \eqref{squared-bessel} which finishes the proof. 

\end{proof}

\section{Proof of Proposition \ref{moment-prop}.}\label{moment-section}

Let us set some notations first. Set $a = \frac{1-\delta}{2}$ and keeping $s,y$ fixed, let $H(s,t,iy) = U_t + i V_t$. Equation \eqref{flow-eqn} then equivalently reads as 
\begin{equation}\label{u,v}
dU_t = dB_t - \frac{aU_t}{U_t^2 + V_t^2}dt, \hspace{2mm} dV_t = \frac{aV_t}{U_t^2 + V_t^2}dt, \hspace{2mm} U_s=0, V_s=y.
\end{equation} 
Following computations are taken from \cite{lawler09} and recited here for readers' convenience, see \cite[Proposition $2.1$]{lawler09} for details. Differentiating \eqref{flow-eqn} both side w.r.t. $z$ shows that 
\begin{equation}\label{derivative-identity}
|H'(s,t,iy)| = \exp\biggl\{\int_s^t\frac{a(U_r^2 - V_r^2)}{(U_r^2 + V_r^2)^2}dr\biggr\}.
\end{equation}

An It\^o formula based computation implies that 

\begin{equation}\label{ito-computation}
\int_s^t\frac{\lambda a(U_r^2 - V_r^2)}{(U_r^2 + V_r^2)^2}dr= M_t  - \frac{1}{2}[M]_t + \log\biggl\{\biggl(\frac{y}{V_t}\biggr)^{\zeta}\biggl(1+ \frac{U_t^2}{V_t^2}\biggr)^{\frac{-\theta}{2}}\biggr\},
\end{equation} 
where $\lambda, \zeta$ and $\theta$ are related by $\lambda = \theta\bigl(1 + \frac{1}{2a}\bigr) - \frac{\theta^2}{4a}$, $\zeta= \theta -\frac{\theta^2}{4a}$ and 
\[M_t = \theta \int_s^t \frac{U_r}{U_r^2 + V_r^2}dB_r.\]\\

We will also need the following Lemma.

\begin{lemma}\label{distortion} If $F: \mathbb{H} \to \mathbb{C}$ is an injective holomorphic map and $z, w \in \mathbb{H}$ with $Im(z), Im(w) \geq y >0 $, then 
\[ |F'(w)| \leq 144^{\frac{|z-w|}{y} + 1}|F'(z)|.\]
\end{lemma}

\begin{proof} See [\cite{lawler-book}-Chapter $4$].
\end{proof}

\begin{lemma}\label{tran-estimate}
We have the following estimates. 
\[|U_t| \leq 2\sup_{r \in [s,t]}|B_r- B_s|,  V_t \leq \sqrt{y^2 + 2a(t-s)}. 
\]
\end{lemma}
\begin{proof}See \cite{RTZ}.
\end{proof}

 \begin{proof}[Proof of Proposition \ref{moment-prop}]

Choose $\theta \in (2, 4a)$. Then $\zeta > 0$ and $\lambda > 2$. Note that $V_t$ is monotonic increasing and $V_t \geq y$. Thus, 
\[\log\biggl\{\biggl(\frac{y}{V_t}\biggr)^{\zeta}\biggl(1+ \frac{U_t^2}{V_t^2}\biggr)^{\frac{-\theta}{2}}\biggr\} \leq 0,\]
and \eqref{derivative-identity},\eqref{ito-computation} implies \eqref{exp-martingale}. For \eqref{weak-tail}, note that by Dambis-Dubins-Schwarz martingale embedding theorem, there exists a Brownian motion $\tilde{B}$ such that $M_t = \tilde{B}_{[M]_t}$. Thus,
\[\sup_{t\in[s,T]}\biggl(M_t - \frac{[M]_t}{2}\biggr) = \sup_{t\in[s,T]}\biggl(\tilde{B}_{[M]_t} - \frac{[M]_t}{2}\biggr) \leq  \sup_{t\in[0, \infty)}\biggl(\tilde{B}_t -\frac{t}{2}\biggr).\]
It is well known that $\sup_{t}\bigl(\tilde{B}_t -\frac{t}{2}\bigr)$ is distributed as an exponential random variable with parameter $1$ and \eqref{exp-martingale} then easily implies \eqref{weak-tail}. \\

The proof of \eqref{crucial-estimate} follows from \eqref{weak-tail} using the following Borel-Cantelli argument. The following argument is similar to the argument in \cite{optimal-Holder}, but the following argument is simpler because we have avoided the use of Bieberbach's Theorem, see \cite[Lemma 3.5]{optimal-Holder}.\\
 Choose a Whitney type discretization of $(s,y)\in [0,T]\times (0,1]$ given by $s_{n,k} = k2^{-2n}T, y_n = 2^{-n}$ for $n\geq 1$ and $1\leq k\leq 2^{2n}$. Then, for $\beta \in (\frac{2}{\lambda},1)$, \eqref{weak-tail} implies  
\[\sum_{n=1}^{\infty}\sum_{k=1}^{2^{2n}}\mathbb{P}\bigl[\sup_{t\in [s_{n,k},T]}|H'(s_{n,k},t,y_n)| \geq 2^{n\beta}\bigr] \leq \sum_{n=1}^{\infty} \frac{1}{2^{(\beta\lambda -2)n}} < \infty.\]
It follows using Borel-Cantelli Lemma that almost surely for $n$ large enough and for all $1\leq k\leq 2^{2n}$, 
\begin{equation}\label{discrete}
\sup_{t\in [s_{n,k},T]}|H'(s_{n,k},t,y_n)| \leq 2^{n\beta}.
\end{equation}

In order to get the uniform in $(s,y)\in[0,T]\times (0,1]$ estimate \eqref{crucial-estimate} from \eqref{discrete}, we will use Lemma \ref{distortion} as follows. If $t-s \leq y_n^2$, then 
\begin{equation}\label{trivial-estimate} |H'(s,t,iy_n)| = \exp\biggl\{\int_s^t\frac{a(U_r^2 - V_r^2)}{(U_r^2 + V_r^2)^2}dr\biggr\} \leq \exp\biggl\{\frac{a(t-s)}{y_n^2}\biggr\} \leq C.  
\end{equation}

For $t-s > y_n^2$, choose the least $k$ such that $s_{n,k}\geq s$. Then using the flow property and \eqref{trivial-estimate} again, 
\begin{equation}
|H'(s,t,iy_n)| = |H'(s_{n,k},t,H(s,s_{n,k}, iy_n))||H'(s,s_{n,k},iy_n)| \leq C  |H'(s_{n,k},t,H(s,s_{n,k}, iy_n))|.
\end{equation}
Using Lemma \ref{distortion}, 
\begin{equation}
|H'(s_{n,k},t,H(s,s_{n,k}, iy_n))| \leq C (144)^{|H(s,s_{n,k}, iy_n)|/y_n} |H'(s_{n,k},t,iy_n)| \leq C (144)^{|H(s,s_{n,k}, iy_n)|/y_n}2^{n\beta} .
\end{equation}
Note using Lemma \ref{tran-estimate} and L\`evy modulus of continuity for Brownian motion that \[Im(H(s,s_{n,k}, iy_n)) \leq \sqrt{y_n^2 + 2a (s_{n,k}-s)} \leq Cy_n,\]
and 
\[|Re(H(s,s_{n,k}, iy_n))| \leq 2 \sup_{r \in [s,s_{n,k}]} |B_r - B_s| \leq C \sqrt{-(s_{s,k}-s)\log(s_{n,k}-s)}\leq C\sqrt{n}y_n,\]
which implies that $|H(s,s_{n,k}, iy_n)|/y_n \leq C\sqrt{n}$ and at the cost of choosing a slightly larger $\beta$, we obtain 
\[|H'(s,t,iy_n)| \leq C2^{n\beta}.\]
To get uniformity in $y$, choose $n$ such that $y_{n+1} \leq y \leq y_n$ and applying Lemma \ref{distortion} again, 
\[ |H'(s,t,iy)| \leq C(144)^{(y_n-y)/y_{n+1}}|H'(s,t,iy_n)| \leq C2^{n\beta},\]
which proves the claim.  

\end{proof}

\section{Proof of Theorem \ref{cont-flow}}

In this section we will prove the existence of a continuous modification of the stochastic flow $H(s,t,x)$ defined above. The estimate \eqref{crucial-estimate} will again serve as a key component in the proof. Another key component in the proof is the continuity of the level zero hitting times for Bessel processes with respect to starting points. Define 
\[T^{s,x}:= \inf\{t \geq s | H(s,t,x) =0\}.\]
In particular $T^{s,0} = s$.
For $x \in \mathbb{R}$ and $t \leq T^{s,x}$, $H(s,t,x)$ are standard Bessel processes with negative dimension $\delta <0$. If follows from standard results on Bessel processes and comparison theorems on ODEs that almost surely $T^{s,x} < \infty$ for all $s,x$ and almost surely for all $x\neq y$, $T^{s,x} \neq T^{s,y}$. We further claim the following proposition. 

\begin{proposition}\label{continuity-hitting-time} Let $\delta< 0$. Then,
\begin{enumerate}
\item Almost surely $T^{s,x} -s \to 0$ as $x \to 0$ for all $s\in [0,\infty)$.

\item The function $(s,x) \mapsto T^{s,x}$ is almost surely jointly continuous in variables $(s,x) \in [0,\infty)\times \mathbb{R}$.
\end{enumerate}
\end{proposition}

\begin{proof}[Proof of Proposition \ref{continuity-hitting-time}-(a)]
W.l.o.g. we let $s$ vary over $s\in [0,1]$. Note that using scaling of Bessel processes, $T^{s,x} - s$ has the same distribution as $x^2(T^{s,1} - s)$ and $T^{s,1}-s$ has same distribution as $T^{0,1}$. It is well known that $T^{0,1}$ has the explicit density given by 
\[ C_{\delta}t^{\frac{\delta}{2}-2}\exp\{-1/2t\},\]
where $C_{\delta}$ is a normalizing constant, see Proposition $1.9$ in \cite{lawler-notes} for a detailed proof of this fact. It then follows that for $p < 1 -\frac{\delta}{2}$, $\mathbb{E}[(T^{0,1})^p] < \infty$. Now let $\epsilon_n = 2^{-n}$. For each $n\geq 1$, define sequences $\{s_{n,m}\}_{m\geq 0}$ by $s_{n,0} = 0$ and $s_{n,m+1} = T^{s_{n,m}, \epsilon_n}$. Choose constants $p \in (1,1- \frac{\delta}{2})$ and $\alpha, \beta$ such that $ \frac{2}{p} < 2\alpha < \beta < 2$. Consider the event 
\[A_n = \bigcup_{m=1}^\infty\{s_{n,m}- s_{n,m-1} > m^{\alpha}\epsilon_n^{\beta}\}.\]
Then using Markov inequality and the fact that $\mathbb{E}[(T^{0,1})^p] < \infty$,
\begin{align*}
\mathbb{P}[A_n] & \leq \sum_{m=1}^{\infty} \mathbb{P}[s_{n,m}- s_{n,m-1} > m^{\alpha}\epsilon_n^{\beta}]  \\
& \leq \sum_{m=1}^{\infty} \frac{\mathbb{E}[(s_{n,m}- s_{n,m-1})^p]}{m^{p\alpha}\epsilon_n^{p\beta}} \\
& = \sum_{m=1}^{\infty} \frac{\epsilon_n^{p(2 -\beta)}\mathbb{E}[(T^{0,1})^p]}{m^{p\alpha}}.
\end{align*}
Thus, $\sum_{n=1}^{\infty}\mathbb{P}[A_n] < \infty$ and Borel-Cantelli implies that almost surely there exists $N(\omega)$ such that for all $n\geq N(\omega)$, events $A_{n}$ do not occur, i.e.

\begin{equation}\label{cont-estimate}
\mbox{For all}\hspace{2mm}  n\geq N(\omega) \hspace{2mm} \mbox{and} \hspace{2mm} m\geq 1, \hspace{2mm} s_{n,m} - s_{n,m-1} \leq m^{\alpha}\epsilon_n^{\beta}.
\end{equation}

Also note using strong Markov property that sequence $\{s_{n,m}- s_{n,m-1}\}_{m\geq 1}$ is an i.i.d. sequence of random variables and 
\[s_{n,m} = \sum_{j=1}^{m}s_{n,j} - s_{n,j-1}\]
is a random walk. Pick an integer $M> 1/\mathbb{E}[T^{0,1}]$ and consider $\{s_{n,M2^{2n}}\}_{n\geq 1}$. We claim that $s_{n,M2^{2n}}$ converges in probability to $M\mathbb{E}[T^{0,1}]$ as $n \to \infty$. This is essentially weak law of large number (LLN), but since collection of i.i.d. random variables $\{s_{n,m}- s_{n,m-1}\}_{m\geq 1}$ is also changing with $n$, original form of weak LLN does not apply. One can however repeat the proof of weak LLN using characterstic functions similarly in this case also to conclude $s_{n,M2^{2n}}\overset{p}{\to} M\mathbb{E}[T^{0,1}] >1$. 
This implies almost surely  there exist a subsequence of $n_{k} \to \infty$ with $s_{n_{k},M2^{2n_{k}}} > 1$. Using \eqref{cont-estimate}, for $n_{k}$ large enough and $1\leq m\leq M2^{2n_{k}}$
\[T^{s_{n_{k},m-1},\epsilon_{n_{k}}} - s_{n_{k},m-1}= s_{n_{k},m} - s_{n_{k},m-1} \leq M^{\alpha}\epsilon_{n_{k}}^{\beta-2\alpha}.\]
Finally note that using comparison principle for ODE solutions, $T^{s,x}-s$ is monotonic increasing as $x$ increases and the limit $T^{s,0+}-s := \lim\limits_{x\to 0+} T^{s,x}-s$ exists. Similarly, if $s\in [s_{n,m-1}, s_{n,m}]$ then $T^{s,0+}-s \leq s_{n,m}-s_{n,m-1}$. If $s\in [0,1]$, then for each $n_{k}$, there exists $1\leq m\leq M2^{2n_{k}}$ such that $s\in [s_{n_{k},m-1}, s_{n_{k},m}]$ which implies $T^{s,0+}-s \leq M^{\alpha}\epsilon_{n_{k}}^{\beta-2\alpha}$. Letting $ n_{k} \to \infty$ implies $T^{s,0+}-s = 0$. By symmetry and similarly as above, $T^{s,0-}-s =0$ as well which concludes the proof. 

\end{proof}

\begin{proof}[Proof of Proposition \ref{continuity-hitting-time}-(b)]
We first prove that for each fixed $s$, $x\mapsto T^{s,x}$ is continuous. Since this argument is same for all $s$, we assume w.l.o.g. $s=0$ and write $T^{0,x} = T^{x}$. The left continuity of $T^x$ follows easily because if $T^{x-}:= \lim_{y\uparrow x} T^{y} < T^{x}$, then by taking the monotone limit as $y\uparrow x$ of solutions $H(0,t,y)$ starting at $y$, we can construct a solution starting from $x$ which hits zero at $T^{x-}$, which is contradiction by the definition of $T^{x}$. For the right continuity of $T^{x}$, let $y\downarrow x$. The by letting the solution $H(0,t,y)$ flow till time $t=T^{x}$, we obtain $T^y -T^x= T^{T^{x}, H(0,T^x, y)}$. Note that $H(0,T^{x},y) \to 0$ as $y\downarrow x$ and it follows from Proposition \ref{continuity-hitting-time}-(a) that $T^{T^{x}, H(0,T^x, y)} \to 0$ as $y\downarrow x$, which completes the proof. \\

Now for continuity of $T^{s,x}$, let $(s_n,x_n) \to (s,x)$. Depending on whether $s_n>s$ or $s_n<s$, we let the solution $H(s_n,t,x_n)$ flow either in backward or forward direction for time $t$ from $s_n$ to $s$. Then the points $x_n$ will move under this flow to new points, call then $y_n$. Then by definition $T^{s_n,x_n} = T^{s,y_n}$. Note that since $s_n$ is infinitesimally close to $s$, the points $y_n$ will be infinitesimally close to $x$. Using the previous part, since $T^{s,y_n} \to T^{s,x}$, this completes the proof.  

\end{proof}
\begin{remark}
The Proposition \ref{continuity-hitting-time} is no longer true if dimension $\delta >0$. For example when $\delta =1$, then $H(0,t,x) = x + B_t$ is a standard Brownian motion started at $x$. Thus $T^{0,x}$ is given by the first hitting time when $B$ hits level $-x$. It is well known that this is a L\'evy subordinator process which in particular has jumps. This is also implicitly related to the fact that when $\delta >0$, $T^{0,1}$ doesn't have finite $p$-th moment for any $p>1$.   
\end{remark}

\begin{proof}[Proof of Theorem \ref{cont-flow}]
In order to prove Theorem \ref{cont-flow}, we will give a candidate for the stochastic flow $\tilde{H}$ which is a modification of $H$ and we will verify that it is continuous. To define it, we first note that it follows from \eqref{crucial-estimate} and an argument similar to the proof of Theorem \ref{main-thm-1}, the limit $\lim_{y\to 0+}H(s,t,iy)$ exists uniformly in $s,t$. We define $\tilde{H}(s,t,0) := \lim_{y\to 0+}H(s,t,iy)$ which is by definition jointly continuous in $s,t$. Now for $x\in \mathbb{R}\setminus {0}$, if $t\leq T^{s,x}$ then define $\tilde{H}(s,t,x)$ simply by solving the ODE \eqref{flow-eqn} which admits a unique solution for $t\leq T^{s,x}$. For $t> T^{s,x}$, define $\tilde{H}(s,t,x) := \tilde{H}(T^{s,x},t,0)$. Checking the continuity of $\tilde{H}(s,t,x)$ follows easily after knowing that $T^{s,x}$ is continuous jointly in $s,x$ as proven in Proposition \ref{continuity-hitting-time}.  
\end{proof}

%

\end{document}